\documentclass[11pt]{article}
\usepackage{amssymb,amsmath,amsthm}
\usepackage{enumerate}
\usepackage{pgf,tikz}
\usepackage{mathrsfs}
\usepackage{rotating}
\usepackage{pdflscape}
\usetikzlibrary{arrows}
\usepackage{appendix}

\newenvironment{packedItem}{
\begin{itemize}
  \setlength{\itemsep}{1pt}
  \setlength{\parskip}{0pt}
  \setlength{\parsep}{0pt}
}{\end{itemize}}

\let\oldmarginpar\marginpar
\renewcommand\marginpar[1]{\-\oldmarginpar[\raggedleft\footnotesize #1]%
{\raggedright\footnotesize #1}}

\newtheorem{theorem}{Theorem}

\newtheorem{obs}[theorem]{Observation}
\newtheorem{lemma}[theorem]{Lemma}

\newtheorem{conjecture}[theorem]{Conjecture}

\numberwithin{equation}{section}

\newcommand{\pr}{\mathfrak{pr}}

\newcommand{\vanish}[1]{}

\usepackage{verbatim}

\begin{document}

\title{Minimum Coprime Labelings of Generalized Petersen and Prism Graphs}

\author{
John Asplund\\
{\small Department of Technology and Mathematics,
Dalton State College} \\
{\small Dalton, GA 30720, USA} \\
{\small jasplund@daltonstate.edu}\\
\\
N. Bradley Fox\\
{\small Department of Mathematics and Statistics, Austin Peay State University} \\
{\small Clarksville, TN 37044} \\
{\small foxb@apsu.edu}
 }

\date{}
\maketitle

\begin{abstract}
A coprime labeling of a graph of order $n$ is an assignment of distinct positive integer labels in which adjacent vertices have relatively prime labels.  Restricting labels to only the set $1$ to $n$ results in a prime labeling.  In this paper, we consider families of graphs in which a prime labeling cannot exist with the goal being to minimize the largest value of the labeling set, resulting in a minimum coprime labeling.  In particular, prism graphs, generalized Petersen graphs with $k=2$, and stacked prism graphs are investigated for minimum coprime labelings.
\end{abstract}

\section{Introduction}







Let $G$ be a simple graph of order $n$ with vertex set $V$.  We denote two adjacent vertices $v,w$ as $v\sim w$. A \textit{coprime labeling} of $G$ is a labeling of $V$ using distinct labels from the set $\{1,\ldots,m\}$ for some integer $m\geq n$ in which adjacent vertices are labeled by relatively prime integers.  If the integers $1,\ldots, n$ are used as the labeling set, the labeling is called a \textit{prime labeling}, and $G$ is a \textit{prime graph} or is simply referred to as \textit{prime}.  For graphs for which no prime labeling exists, our goal is to minimize the value of $m$, the largest label in the coprime labeling.  This smallest possible value $m$ for a coprime labeling of $G$, denoted by $\pr(G)$, is the \textit{minimum coprime number} of $G$, and a coprime labeling with $\pr(G)$ as the largest label is a \textit{minimum coprime labeling} of $G$.  A prime graph would have a minimum coprime number of $\pr(G)=n$.

Prime labelings of graphs were developed by Roger Entringer and first introduced by Tout, Dabboucy, and Howalla~\cite{TDH}.  Numerous classes of graphs over the past forty years have been shown to be prime, as well as many classes for which a prime labeling has been shown to not exist.  A summary of these results can be seen in Gallian's dynamic survey of graph labelings~\cite{Gallian}.  Most of our upcoming results center around the concept of minimum coprime labelings, which were first studied by Berliner et al.~\cite{Berliner} with their investigation of complete bipartite graphs of the form $K_{n,n}$.  
Asplund and Fox~\cite{AF} continued this line of research by determining the minimum coprime number for classes of graphs such as complete graphs, wheels, the union of two odd cycles, the union of a complete graph with a path or a star, powers of paths and cycles, and the join of paths and cycles.  Recently, Lee~\cite{Lee} made further progress on the minimum coprime number of the join of paths and complete bipartite graphs, in addition to investigating minimum coprime numbers of random subgraphs.

The focus of this paper is to determine the minimum coprime number of prism graphs, which are equivalent to the Cartesian product of a cycle of length $n$ and a path with $2$ vertices, denoted as $C_n\square P_2$. Additionally, a prism graph is equivalent to the generalized Petersen graphs when $k=1$.  In the next section, we include preliminary material regarding the classes of graphs we will investigate and previous research on prime labelings of these graphs.  In Section 3, we construct minimum coprime labelings of prism graphs $GP(n,1)$  for several specific cases of odd $n$ as well as present a conjecture for all sizes of odd prisms.  Section 4 includes results on the minimum coprime number of the generalized Petersen graph $GP(n,2)$, a graph which is not prime for any value $n$.  Section 5 consists of results on minimum coprime number of stacked prism graphs, and finally we investigate a variation of a generalized Petersen graph in Section 6.

\section{Preliminary Material}
An important feature of a graph $G$ that aides in determining whether a prime labeling may exist or if a minimum coprime labeling should instead be investigated is its \textit{independence number}, denoted by~$\alpha(G)$.  Since even number labels must be assigned to independent vertices, the following criteria, first introduced in~\cite{FH}, eliminates the possibility of a prime labeling on many classes of graphs.

\begin{lemma}{\normalfont \cite{FH}}\label{ind}
If $G$ is prime, then the independence number of $G$ must satisfy $\alpha(G)\geq \left\lfloor \frac{|V(G)|}{2}\right\rfloor$.
\end{lemma}

The generalized Petersen graph, denoted $GP(n,k)$ where $n\geq 3$ and $1\leq k \leq \lfloor (n-1)/2\rfloor$, consists of $2n$ vertices $v_1,\ldots, v_n,u_1,\ldots, u_n$.  It has $3n$ edges described by $v_i\sim v_{i+1}$, $u_i\sim u_{i+k}$, and $v_i\sim u_i$ where indices are calculated modulo $n$.  In the particular case of $k=1$, the two sets of vertices form $n$-gons that are connected to form a prism graph, which will be our first graph that we investigate.

When $n$ is odd, $GP(n,1)$ consists of two odd cycles connected by a perfect matching.  Only $(n-1)/2$ vertices on each cycle can be independent, hence $\alpha(GP(n,1))=n-1$ for odd $n$.  Then by Lemma~\ref{ind}, $GP(n,1)$ is not prime in this case, a property which extends to any value of $k$ when $n$ is odd. In fact, $GP(n,k)$ was proven to not be prime in~\cite{PG} for any odd value of $n$ as well as when $n$ and $k$ are both even. Independence numbers for generalized Petersen graphs for certain cases have been determined~\cite{BHM,BEA,FGS} that help provide bounds for the minimum coprime numbers of $GP(n,k)$ in the non-prime cases.

The remaining case of $GP(n,k)$ with $n$ even and $k$ odd is conjectured to be prime for all such $n$ and $k$.  When $k=1$, the prism graph $GP(n,1)$ has been proven to be prime in many specific cases in~\cite{HLYZ} such as when $2n+a$ or $n+a$ are prime for several small values of $a$.  Additional cases of $GP(n,1)$ were proven to be prime in~\cite{PG}.  Dean~\cite{dean} proved the conjecture that all ladders are prime.  Since ladders are simply prism graphs with two edges removed, one might expect his prime labeling to carry over to $GP(n,1)$.  However, when applying this labeling to $GP(n,1)$, these two additional edges do not maintain the relatively prime condition for all~$n$.

While some results have been found on $GP(n,3)$ in~\cite{HLYZ2}, most work involving prime labelings of the generalized Petersen graph has been focused on the prism graph. In~\cite{SSCEHRW} a number theoretic conjecture was made to bolster the conjecture that $GP(n,1)$ is prime for all even $n$. Conjecture 2.1 in~\cite{SSCEHRW} stated that for any even integer $n$, there exists an $s\in [1,n-1]$ such that $n+s$ and $2n+s$ are prime.  By verifying this conjecture for all even $n$ up to $2.468\times 10^9$, they demonstrated $GP(n,1)$ is prime with even $n$ up to that value.  

We conclude this section with the following observations, which will be used without citation in many of the theorems throughout this paper.

\begin{obs}\label{easy}
For positive integers $a,b$, and $k$, the following hold true: 
\begin{packedItem}
\item $\gcd\{a,b\}=\gcd\{ka,b\}$.
\item $\gcd\{a,b\}=\gcd\{a-b,b\}$.
\item $\gcd\{a,b\}=\gcd\{a+b,b\}$.
\item If $a+b$ is prime, then $\gcd\{a,b\}=1$. 
\item If $a-b$ is a prime $p$ and both $a$ and $b$ are not multiples of $p$, then $\gcd\{a,b\}=1$. 

\end{packedItem}
\end{obs}

\section{Prism Graphs}

In this section, we provide several specific results for the minimum coprime number of $GP(n,1)$ and conjecture that $\pr(GP(n,1))=2n+1$ for all odd $n$. Many of the theorems follow a similar proof strategy, but a general construction extended from our techniques seems unlikely without the resolution of longstanding number theory conjectures.  See Figure~\ref{gp11_1} for an example of our first result showing a minimum coprime labeling for the prism $GP(11,1)$.

\begin{theorem}\label{n}
If $n$ is prime, then $\pr(GP(n,1))=2n+1$.
\end{theorem}

\begin{figure}[htb]
    \centering
    \includegraphics[width=2in]{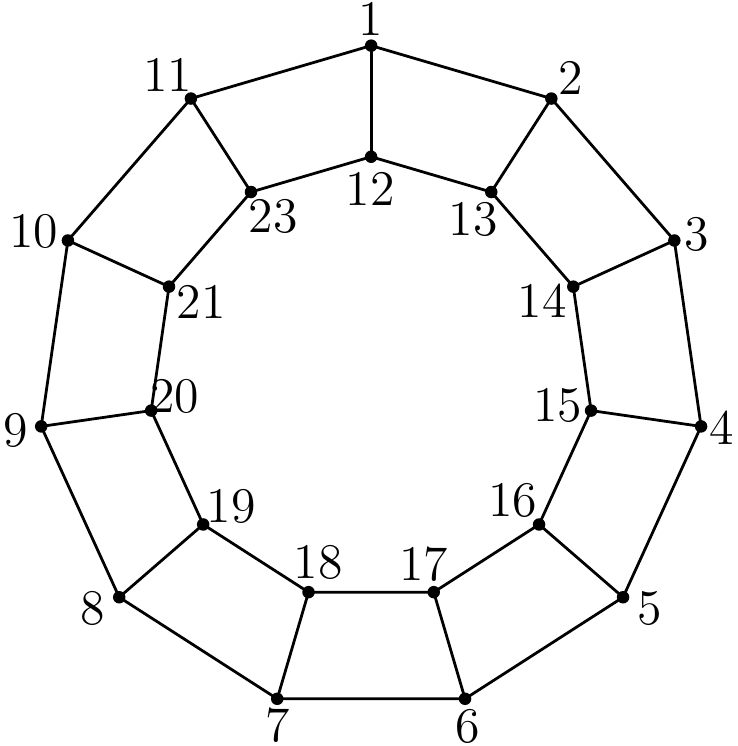}
    \caption{Minimum coprime labeling for GP(11,1)}
    \label{gp11_1}
\end{figure}

\begin{proof}
Label $v_1,\ldots, v_n$ as $1,\ldots, n$ and $u_1,\ldots, u_n$ with $n+1,\ldots, 2n-1, 2n+1$ respectively. All adjacent pairs in $\{v_1,\ldots,v_n\}$ and in $\{u_1,\ldots,u_n\}$ have  consecutive labels except for $v_1v_n$, $u_{n-1}u_n$, and $u_1u_n$. The first pair includes $1$ as one of the labels, and the second pair is labeled by consecutive odd labels.  Lastly, $u_1u_n$ have relatively prime labels since $\gcd\{n+1,2n+1\}=\gcd\{2n+2,2n+1\}=1$.  

It remains to show that the labels on $u_i$ and $v_i$ are relatively prime for each $i$. For $i\leq n-1$, the difference between the labels on $v_i$ and $u_i$ is $n$.  Since $n$ is assumed to be prime, these pairs are relatively prime by Observation~\ref{easy}. Finally when $i=n$, we have $\gcd\{2n+1,n\}=\gcd\{n+1,n\}=1$.  Therefore, this is a coprime labeling, hence $\pr(GP(n,1))\leq 2n+1$.  Since the independence number of $GP(n,1)$ is $n-1$ when $n$ is odd, a prime labeling is not possible, making $\pr(GP(n,1))>2n$.  Thus when $n$ is prime, we have $\pr(GP(n,1))=2n+1$.
\end{proof}

\begin{theorem}\label{n+2}
If $n+2$ is prime, then $\pr(GP(n,1))=2n+1$.
\end{theorem}

\begin{proof}
We will construct a labeling $\ell$ in the following manner. Label $v_1,\ldots,v_n$ with the numbers  $1,\ldots,n$, respectively, and the vertices $u_1,\ldots,u_n$ with the labels $n+3,\ldots, 2n+1, n+2$, respectively.   


Edges between vertices in $\{v_1,\ldots,v_n\}$ have vertices with consecutive labels or contain the label~$1$, and so
$\gcd\{\ell(v_i),\ell(v_{i+1})\}=\gcd\{\ell(v_1),\ell(v_n)\}=1$. Edges between vertices in $\{u_1,\ldots,u_n\}$ have vertices with consecutive labels or with labels $n+2$ and $2n+1$ in which $n+2$ is prime.  Hence the labels between pairs of adjacent vertices in $\{u_1,\ldots,u_n\}$ are relatively prime. For $i=1,\ldots, n-1$, since the difference of the labels on vertices $u_i$ and $v_i$ is $n+2$, which is prime, $\gcd\{\ell(u_i),\ell(v_i)\}=1$. Lastly, $u_n$ and $v_n$ are labeled by consecutive odd integers.  Thus, we have a coprime labeling that is minimal since $GP(n,1)$ is not prime for odd $n$.  Thus, $\pr(GP(n,1))=2n+1$ assuming $n+2$ is prime.
\end{proof}

\begin{theorem}\label{2n+1}
If $2n+1$ is prime, then $\pr(GP(n,1))=2n+1$.
\end{theorem}

\begin{proof}
We construct a labeling $\ell$ by first labelling $v_1,\ldots, v_n$ as $1,\ldots, n$ as in last theorem, but we label $u_1,\ldots, u_n$ in reverse order as $2n+1, 2n-1, 2n-2,\ldots,n+1$.

Edges connecting vertices in $\{v_1,\ldots,v_n\}$ have consecutive labels or contain the label $1$, and so $\gcd\{\ell(v_i),\ell(v_{i+1})\}=\gcd\{\ell(v_1),\ell(v_n)\}=1$. 
Edges between vertices in $\{u_1,\ldots,u_n\}$ have vertices with consecutive labels, with  consecutive odd integer labels, or with the pair of labels $n+1$ and $2n+1$, and so the labels on these adjacent vertices are relatively prime. 
For each $i=2,\ldots, n$, since $\ell(u_i)+\ell(v_i)=2n+1$, we know that $\gcd\{\ell(u_i),\ell(v_i)\}=1$ by Observation~\ref{easy} because $2n+1$ is prime.  Finally, the edge $u_1v_1$ includes the label $1$ on $u_1$.  Thus, this is a coprime labeling that shows $\pr(GP(n,1))=2n+1$ when $2n+1$ is prime. 
\end{proof}

\begin{theorem}\label{2n-1}
If $2n-1$ is prime, then $\pr(GP(n,1))=2n+1$.
\end{theorem}

\begin{proof}
Notice that when $n\equiv 2\pmod{3}$, then $2n-1$ is divisible by 3, so we assume $n\not\equiv 2\pmod{3}$.
If $n\equiv 0\pmod{3}$, then we use the labeling defined in Table~\ref{tab:2n-1} where the top row represents $v_1,\ldots,v_n$ and the bottom row represents $u_1,\ldots,u_n$. One can see the pairs $u_iu_{i+1}$ and $v_iv_{i+1}$ have relatively prime labels, where $\gcd\{\ell(u_1),\ell(u_n)\}=\gcd\{2n+1,2n-2\}=1$ since $n$ is a multiple of $3$ in this case.  The pairs $u_iv_i$ have relatively prime labels for trivial reasons or because their sum is $2n+1$, which is assumed to be prime.

If $n\equiv 1\pmod{3}$, then use the same labeling in Table~\ref{tab:2n-1} except $2n-2$ is labeled as $2n$ instead.  This is a coprime labeling for similar reasoning as our first case, except now the pair of labels $2n$ and $2n-3$ on $u_1$ and $u_2$, respectively, are relatively prime since $n\equiv 1\pmod{3}$.  In each case, the labeling is a minimum coprime labeling, proving $\pr(GP(n,1))=2n+1$ assuming $2n-1$ is prime.
\begin{table}[htb]
    \centering
    \begin{tabular}{|c|c|c|c|c|}
    \hline
         1& 2  & $\cdots$ & $n-1$ & $2n-1$   \\
    \hline
         $2n-2$ & $2n-3$ & $\cdots$  & $n$ & $2n+1$  \\
    \hline
    \end{tabular}
    \caption{Labeling for Theorem~\ref{2n-1}}
    \label{tab:2n-1}
\end{table}
\end{proof}


For further results regarding minimum coprime labelings of prism graphs in other specific cases, see Appendix A.  Using Theorems~\ref{n}-\ref{2n-1}, along with Theorems~\ref{n+4}-\ref{n+6} in the Appendix, an explicit minimum coprime labeling is given for $GP(n,1)$ for all $n$ up to 1641. The following is a more general construction assuming a particular pair of prime numbers exists.


\begin{theorem}\label{conj}
Let $n\geq 3$ be odd.  If there exists an $s\in [3,n-1]$ such that $n+s+1$ and $2n+s+2$ are prime, then $\pr(GP(n,1))= 2n+1$. 
\end{theorem}
\begin{proof}
We use the labeling defined in Table~\ref{s-conj} where the top row represents the vertices $v_1,\ldots,v_n$ and the bottom row represents the vertices $u_1,\ldots,u_n$.  The vertex pairs on edges of the form $u_1v_1$, $u_iu_{i+1}$ and $v_iv_{i+1}$ either contains the label 1, are consecutive integers, are consecutive odd integers, or are the relatively prime pair $n+1$ and $2n+1$.  The adjacent pairs $u_iv_i$ for $i=2,\ldots, s$ have labels that add to $n+s+1$, and the pairs $u_iv_i$ for $i=s+1, \ldots, n$ are labeled by integers whose sum is $2n+s+2$.  Since both of these sums are assumed to be prime, the labels on those pairs are relatively prime as well.  Thus, since $GP(n,1)$ is not prime when $n$ is odd, we have constructed a minimum coprime labeling proving that $\pr(GP(n,1))=2n+1$ if such a value $s$ exists.
\end{proof}

\begin{table}[htb]
    \centering
    \begin{tabular}{|c|c|c|c|c|c|c|c|c|c|}
    \hline
         1& 2 & $\cdots$ & $s-1$ & $s$ & $s+1$ & $s+2$ & $\cdots$ & $n-1$ & $n$ \\
    \hline
         $n+s+1$ & $n+s-1$ & $\cdots$ & $n+2$ & $n+1$ & $2n+1$ & $2n$ & $\cdots$ & $n+s+3$ & $n+s+2$  \\
    \hline
    \end{tabular}
    \caption{Labeling for Theorem~\ref{conj}}
    \label{s-conj}
\end{table}

Recall that Conjecture 2.1 in~\cite{SSCEHRW} states that for all even integers $n$, there is an $s<n$ such that $n+s$ and $2n+s$ are both prime.  If this is true for all even integers, then the previous theorem would prove the subsequent conjecture for all odd $n$ since applying Conjecture 2.1 to the even integer $n+1$ would result in $n+s+1$ and $2n+s+2$ being prime.  Results in~\cite{SSCEHRW} combine with Theorem~\ref{conj} to confirm the following conjecture for odd $n<2.468\times 10^9$.

\begin{conjecture}
For all odd $n\geq 3$, $\pr(GP(n,1))=2n+1$.
\end{conjecture}

\section{Generalized Petersen Graphs with $k=2$}
We next consider the generalized Petersen graph in the case of $k=2$.
The vertices of $GP(n,2)$ with $n\geq 5$ are still referred to as $v_1,\ldots,v_n$,$u_1,\ldots,u_n$ with the edges of the forms $v_i v_{i+1}$, $u_i u_{i+2}$, and $v_i u_i$ in which indices calculated modulo $n$.

The independence number for generalized Petersen graphs when $k=2$ is given by the formula~$\lfloor\frac{4n}{5}\rfloor$, as shown in~\cite{BHM}
through their study of minimum vertex covers of $GP(n,2)$.  This results in the generalized Petersen graph with $k=2$ not being prime for any value of $n$.  The denominator of this formula for the independence number provides a natural direction by which to create an independent set for this graph by dividing $GP(n,2)$ into blocks that include $5$ of the $v_i$ and $5$ of the $u_i$ vertices. 
We utilize this technique in the following proof but limit ourselves for now to the case when $n$ is a multiple of $5$.

\begin{lemma}\label{GPn2Lemma}
Let $m$ be a positive integer. Then $\pr(GP(5m,2))=12m-1$.
\end{lemma}

\begin{proof}
We aim to construct a coprime labeling $\ell$ and later will show that it is minimal. 
We begin by assigning $v_1,\ldots,v_5$ the labels  $2,3,5,8,9$ and assigning $u_1,\ldots,u_5$ the labels $1,4,6,7,11$, respectively. 
One can verify these ten labels form a coprime labeling when $m=1$.  For $m>1$ we then define the following labeling for the block of ten vertices $v_{5k+1},\ldots,v_{5k+5},u_{5k+1},\ldots,u_{5k+5}$ for each $1\leq k<m$
\begin{align}\label{initiallabel}
\begin{array}{ccccccccccc}
    \ell(v_{5k+1}) = 12k+2 && \ell(v_{5k+4}) = 12k+8 && \ell(u_{5k+1}) = 12k+1 && \ell(u_{5k+4}) = 12k+7    \\
    \ell(v_{5k+2})= 12k+3 && \ell(v_{5k+5}) = 12k+9 && \ell(u_{5k+2})= 12k+4 &&  \ell(u_{5k+5}) = 12k+11. \\
   \ell(v_{5k+3}) = 12k+5 &&  && \ell(u_{5k+3}) = 12k+10 && 
\end{array}
\end{align}

\begin{figure}[htb]
    \centering
    \includegraphics[width=5in]{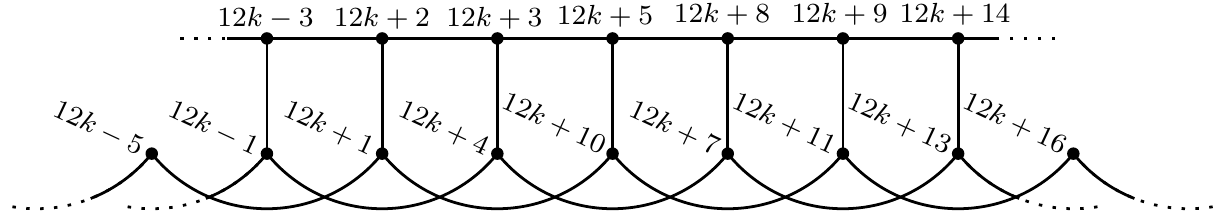}
    \caption{Visual representation of the labeling described in Equation~\eqref{initiallabel}}
    \label{fig:gp5m2}
\end{figure}

See Figure~\ref{fig:gp5m2} for a visual representation of the labeling of the ten vertices in Equation~\eqref{initiallabel} and their adjacent vertices.
The labeling $\ell$ as currently defined is not enough to guarantee each pair of adjacent vertices has relatively prime labels, particularly for pairs of labels that have a difference of~$5$. 
We alter the labeling $\ell$ by addressing cases for specific~$k$ values based on the divisibility of $12k-1$, $12k-3$, and $12k+5$. 
Before altering $\ell$, first note that no adjacent vertices are both labeled by even integers.  
One can also observe that no labels that are multiples of $3$ are assigned to adjacent vertices, including the adjacent pairs whose labels differ by 9.  Additionally, the final vertices in the last block $u_{n-1}$, $u_n$, and $v_n$ are adjacent to the vertices $u_1$, $u_2$, and $v_1$, respectively.  Since $\ell(u_1)=1$, it is relatively prime to the label of $u_{n-1}$.  Likewise, $\ell(u_2)=4$ and $\ell(v_1)=2$, while $\ell(u_{n})=12n-1$ and $\ell(v_n)=12n-3$ are both odd, making those adjacent pairs of labels also relatively prime.

As we define $\ell$ for the upcoming cases, the labels on vertices $u_{5k+4}, u_{5k+5},$ and $v_{5k+5}$ are not changed except in Cases 4b, 4c, and 4d, and this occurs only as the subsequent block is labeled. These three vertices are the only ones within the block of ten vertices $v_{5k+1},\ldots,v_{5k+5},u_{5k+1},\ldots,u_{5k+5}$ that are adjacent to vertices in the subsequent block, hence leaving these three vertices unchanged is essential to guaranteeing that adjacent labels on vertices in different blocks are relatively prime. Let $U_k=\{u_{5k+1},\ldots,u_{5k+5}\}$ and $V_k=\{v_{5k+1},\ldots,v_{5k+5}\}$. 

\noindent \textbf{Case 1:} Suppose that $5\nmid 12k-1,5\nmid 12k-3,5\nmid 12k+5$.

Label the vertices $U_k\cup V_k$ as in Equation~\ref{initiallabel}. As previously observed, pairs of adjacent vertices in $U_k\cup V_k$ or adjacent pairs between the vertices in $U_k\cup V_k$ and $\{v_{5k},u_{5k},u_{5k-1}\}$ do not have labels that share a common factor of 2 or 3.  The adjacent vertex pairs $\{u_{5k},u_{5k+2}\}$,  $\{v_{5k}, v_{5k+1}\}$, and $\{v_{5k+3},u_{5k+3}\}$ have labels that differ by $5$.  Our assumptions for this case ensure that these pairs are not both divisible by $5$, resulting in the relatively prime condition being satisfied.


\noindent \textbf{Case 2:} Assume that $5\mid 12k+5$.

Use the labeling $\ell$ from Equation~\eqref{initiallabel} except we redefine
\begin{align}
\label{case2}
    \begin{array}{c}
         \ell(u_{5k+3})=12k+6.
    \end{array}    
\end{align}
Since we assumed $5\mid 12k+5$, it follows that $5\nmid 12k-1$ and $5\nmid 12k-3$, and thus after applying reasoning from Case 1, we need only check that $\ell(u_{5k+3})$ is relatively prime with the labels of all neighbors of $u_{5k+3}$. Since $u_{5k+3}$ is adjacent to $u_{5k+1}$, $u_{5k+5}$ and $v_{5k+3}$, we need that $\gcd\{12k+6,12k+1\}=1$, $\gcd\{12k+6,12k+11\}=1$, and $\gcd\{12k+6,12k+5\}=1$. The third equality is trivial, and the first two equalities follow immediately from the Case 2 assumption.  

\noindent \textbf{Case 3a:} Next, suppose that $5\mid 12k-1$ and $7\nmid 12k-3$.

Use the initial labeling $\ell$ in Equation~\eqref{initiallabel} with the following two altered labels
\begin{align}
\label{case3a}
    \begin{array}{cccc}
    \ell(u_{5k+2})=12k+2,     & \ell(v_{5k+1})=12k+4.
    \end{array}
\end{align}
Notice that since $5$ divides $12k-1$, we have $\gcd\{12k+10,12k+5\}=1$. 
As before, we need only to check that $\ell(u_{5k+2})$ and $\ell(v_{5k+1})$ are relatively prime with the labels of any adjacent vertices.
Clearly, $\gcd\{12k+2,12k+3\}=\gcd\{12k+4,12k+3\}=1$. Since both $12k+2$ and $12k+4$ are not divisible by $3$, we know that $\gcd\{12k+2,12k-1\}=\gcd\{12k+4,12k+1\}=1$. Since 7 is assumed to not divide $12k-3$, $\gcd\{12k+4,12k-3\}=1$. Finally, our assumption of $5\mid 12k-1$ implies $5\nmid 12k+2$, hence $\gcd\{12k+2,12k+7\}=1$.

\noindent \textbf{Case 3b:} We now assume $5\mid 12k-1$ and $7\mid 12k-3$.

We reassign the following labels from $\ell$ 
\begin{align}
\label{case3b}
    \begin{array}{ccccccc}
        \ell(u_{5k+2})=12k+6, && \ell(v_{5k+2})=12k+5, && \ell(v_{5k+3})=12k+3.
    \end{array}
\end{align}
We need only check that these new labels are relatively prime with labels of any neighboring vertices.
It is clear that 
\[
\gcd\{12k+6,12k+7\}=\gcd\{12k+6,12k+5\}=\gcd\{12k+5,12k+3\}=1.
\]
Since $12k+2$ is not divisible by 3, $\gcd\{12k+2,12k+5\}=1$.
By our assumption that $12k-1$ is divisible by 5, $\gcd\{12k+3,12k+8\}=1$. 
Since $12k-3$ is assumed to be divisible by 7, $\gcd\{12k-1,12k+6\}=\gcd\{12k+3,12k+10\}=1$.

\noindent \textbf{Case 4a:} Suppose that $5\mid 12k-3$ and $7\nmid 12k-3$.

We make the following changes to $\ell$ 
\begin{align}
 \label{case4a}
    \begin{array}{cccccccccc}
        \ell(v_{5k+1})=12k+4, && \ell(v_{5k+4})=12k+10, && \ell(u_{5k+2})=12k+8, && \ell(u_{5k+3})=12k+2.
    \end{array}
\end{align}
Clearly we have 
\[
\gcd\{12k+4,12k+3\}=\gcd\{12k+2,12k+1\}=\gcd\{12k+8,12k+7\}=\gcd\{12k+10,12k+9\}=1.
\]
Additionally, since none of the four reassigned labels are divisible by 3, it is clear that 
\begin{align*}
\gcd\{12k+4,12k+1\}&=\gcd\{12k+2,12k+5\}=\gcd\{12k+10,12k+7\}\\
&=\gcd\{12k+8,12k-1\}=\gcd\{12k+2,12k+11\}=1.
\end{align*}
Our assumptions in this case include that $7\nmid 12k-3$ and also imply that $5\nmid 12k+8$ or $12k+10$.  Thus we have
\[
\gcd\{12k+4,12k-3\}=\gcd\{12k+8,12k+3\}=\gcd\{12k+10,12k+5\}=1.
\]

\noindent \textbf{Case 4b:} Now suppose that $5\mid 12k-3$, $7\mid 12k-3$, and $11\nmid 12k-3$.

Once again, we make four changes to $\ell$ in this case
\begin{align}
\label{case4b}
    \begin{array}{cccccccccc}
        \ell(v_{5k})=12k-1, && \ell(u_{5k})=12k-3, && \ell(u_{5k+2})=12k+8, && \ell(v_{5k+4})=12k+4.
    \end{array}
\end{align}
Since the two vertices indexed by $5k$ are in the previous block of ten vertices, it is important to consider whether that block falls within a case in which any labels were swapped from the initial labeling of that block.  Since we assume $5\mid 12k-3$, $5 \mid 12k-13=12(k-1)-1$, whereas $7\mid 12k-3$ implies $7\nmid 12k-15=12(k-1)-3$ so vertices in $U_{k-1}\cup V_{k-1}$ would be labeled according to Case 3a.  Neither vertex whose label was swapped within Case 3a is adjacent to $v_{5k}$ or $u_{5k}$, so the adjacent pairs of labels to consider from that block are $12k-3$ and $12k-2$, $12k-3$ and $12k-1$, and $12k-1$ and $12k-4$.  

Overall, there are ten adjacent pairs of labels that need to be verified as relatively prime.  It is clear that 
\[
\gcd\{12k-3,12k-2\}=\gcd\{12k-3,12k-1\}=\gcd\{12k+8,12k+7\}=\gcd\{12k+4,12k+5\}=1.
\]
Since the only reassigned label that is a multiple of $3$ is $12k-3$, we have 
\[
\gcd\{12k-1,12k-4\}=\gcd\{12k-1,12k+2\}=\gcd\{12k+4,12k+7\}=1.
\]
The assumption $5\mid 12k-3$ implies $5$ is not a factor of $12k+8$ or $12k+4$, hence
\[
\gcd\{12k+8,12k+3\}=\gcd\{12k+4,12k+9\}=1.
\]
Finally, by our assumption that $11\nmid 12k-3$, we know $\gcd\{12k-3,12k+8\}=1$.


\noindent \textbf{Case 4c:} Assume that $5\mid 12k-3$, $7\mid 12k-3$, $11\mid 12k-3$, and $13\nmid 12k-3$.

Five reassignments of labels are needed
\begin{align}
\label{case4c}
    \begin{array}{cccccc}
        \ell(v_{5k})=12k-1, && \ell(u_{5k})=12k-3, && \ell(u_{5k+2})=12k+10, \\
        \ell(u_{5k+3})=12k+8, && \ell(v_{5k+4})=12k+4. &&
    \end{array}
\end{align}
The reason why we can change the labels on $v_{5k}$ and $u_{5k}$ without causing any adjacent pairs of vertices to not be relatively prime is the same reason as given in Case 4a.
There are ten additional pairs of labels that need to be shown to be relatively prime to complete this case.  Clearly,  
\begin{align*}
\gcd\{12k-1,12k+2\} &= \gcd\{12k+4,12k+5\} = \gcd\{12k+10, 12k+7\}\\ 
&= \gcd\{12k+8,12k+5\}=\gcd\{12k+8,12k+11\}\\
&=\gcd\{12k+4,12k+7\} =1.
\end{align*}
By our assumption that $5\mid 12k-3$, we know $5$ is not a factor of $12k+4$, resulting in $\gcd\{12k+4,12k+9\}=1$.  Similarly, $7$ is assumed to be a factor of $12k-3$, so $7\nmid 12k+3$ and $7\nmid 12k+8$; therefore, $\gcd\{12k+10,12k+3\}=\gcd\{12k+8,12k+1\}=1$.  Lastly, we assumed $13\nmid 12k-3$, hence $\gcd\{12k-3,12k+10\}=1$, resulting in the relatively prime condition being satisfied.








\noindent \textbf{Case 4d:} Finally we suppose that $5\mid 12k-3$, $7\mid 12k-3$, $11\mid 12k-3$, and $13\mid 12k-3$.

In this case, we only need one pair of labels to be swapped by relabeling $u_{5k-2}$ as $12k+2$ and $v_{5k+1}$ as $12k-2$.  As in the last two cases, the previous block of vertices that contains $u_{5k-2}$ falls within Case 3a, which involves swapping two labels on vertices that are not adjacent to $u_{5k-2}$.  Its neighbors then are labeled by $12k-11$, $12k-7$, and $12k-1$, while the label $12k-2$ on $v_{5k+1}$ is adjacent to $12k-3$, $12k+1$, and $12k+3$.  
Thus, by our assumptions in this case, our reassigned labels are relatively prime with the labels of adjacent vertices.

Therefore, by our assumptions and case analysis, it is clear that all labels are relatively prime with their adjacent labels.  Thus in each case the updated $\ell$ is a coprime labeling, making $\pr(GP(5m,2))\leq 12m+1$.
Since $\alpha(GP(5m,2))=\left\lfloor\frac{4(5m)}{5}\right\rfloor=4m$, we need $6m$ odd numbers to label the vertices in $GP(5m,2)$. Thus, $\pr(GP(5m,2))\geq 12m-1$. Therefore, $\pr(GP(5m,2))=12m-1$.
%
\end{proof}

\begin{theorem}
The minimum coprime number for $GP(n,2)$ for $n\geq 5$ is given by
$$\pr(GP(n,2))=
\begin{cases}
12m-1 & \text{if } n=5m\\
12m+3 & \text{if } n=5m+1\\
12m+5 & \text{if } n=5m+2\\
12m+7 & \text{if } n=5m+3\\
12m+9 & \text{if } n=5m+4.\\
\end{cases}$$
\end{theorem}

\begin{proof}
When $n=5m$, we constructed in Lemma~\ref{GPn2Lemma} a minimum coprime labeling $\ell$ with $12m-1$ as the largest label.  
For the remaining cases, we will build the labeling by using $\ell$ defined in Lemma~\ref{GPn2Lemma} for the first $5m$ vertices in $v_1\ldots v_n$ and the first $5m$ vertices in $u_1\ldots u_n$. Note that the vertices $v_{5m}$, $u_{5m}$, and $u_{5m-1}$ are not changed from the labeling defined in Equation~\eqref{initiallabel}. Hence, $\ell(v_{5m})=12m-3$, $\ell(u_{5m})=12m-1$, and $\ell(u_{5m-1})=12m-5$.

Suppose that $n=5m+1$. Label the remaining vertices as $\ell(u_{5m+1})=12m+3$ and $\ell(v_{5m+1})=12m+1$. 
By Lemma~\ref{GPn2Lemma}, we need only check that the following pairs of adjacent vertices have relatively prime labels: $u_{5m+1} v_{5m+1}$, $u_{5m} u_1$, $u_{5m+1} u_2$, $v_{5m+1} v_1$, $u_{5m-1} u_{5m+1}$, and $v_{5m} v_{5m+1}$.  
It is clear we have each of the following necessary relatively prime pairs 
\begin{align*}
\gcd\{12m-1,1\}&=\gcd\{12m+1,2\}= \gcd\{12m+3,4\}=\gcd\{12m+1,12m+3\}\\
&=\gcd\{12m-5,12m+3\}=\gcd\{12m-3,12m+1\}=1.
\end{align*}
Since the independence number is $\alpha(GP(5m+1,2))=\left\lfloor \frac{4(5m+1)}{5}\right\rfloor=4m$, we have used the maximum number of even labels less than $12m+3$. Since all odd integers were used from 1 to $12m+3$, we have that $\pr(GP(n,2))=12m+3$ in the case of $n=5m+1$. 

Next, suppose that $n=5m+2$.  We label the vertices $v_1,\ldots,v_{5m}$ and $u_1,\ldots,u_{5m}$ as in Lemma~\ref{GPn2Lemma}.
Label the remaining vertices as 
\begin{align*}
    \begin{array}{cccccccccc}
        \ell(u_{5m+1})=12m+4, && \ell(u_{5m+2})=12m+5, && \ell(v_{5m+1})=12m+1, && \ell(v_{5m+2})=12m+3. \\
    \end{array}
\end{align*}
As explained above, $\ell(u_{5m})=12m-1$, $\ell(v_{5m})=12m-3$, and $\ell(u_{5m-1})=12m-5$.
Also note that our new labels on vertices adjacent $v_1$, $u_1$, and $u_2$ make relatively prime pairs since $\ell(u_{5m+2})$ and $\ell(v_{5m+2})$ are odd.
For the remaining adjacent pairs, we have
\begin{align*}
\gcd\{12m+4,12m-5\}&=\gcd\{12m+4,12m+1\}= \gcd\{12m+5,12m-1\}\\ 
&=\gcd\{12m+5,12m+3\}=\gcd\{12m+1,12m-3\}\\
&=\gcd\{12m+1,12m+3\}=1.
\end{align*}
The independence number in this case is $\alpha(GP(5m+2,2))=\left\lfloor\frac{4(5m+2)}{5}\right\rfloor=4m+1$, which shows we used the maximum number of even labels since one of the last four vertex labels is even.  Thus, $\pr(GP(n,2))=12m+5$ when $n=5m+2$.

Assume next that $n=5m+3$.  Again we label the vertices $v_1,\ldots,v_{5m}$ and $u_1,\ldots,u_{5m}$ as in Lemma~\ref{GPn2Lemma}.
Label the remaining vertices as
\begin{align*}
    \begin{array}{ccccccc}
         \ell(u_{5m+1})=12m+4, && \ell(u_{5m+2})=12m+3, && \ell(u_{5m+3})=12m+7, \\
         \ell(v_{5m+1})=12m+1, && \ell(v_{5m+2})=12m+2, && \ell(v_{5m+3})=12m+5. \\
    \end{array}
\end{align*}
Since $\ell(v_{5m+3})$ and $\ell(u_{5m+3})$ are odd, they are relatively prime with $\ell(v_1)$ and $\ell(u_2)$, respectively. The remaining adjacent pairs satisfy the following
\begin{align*}
\gcd\{12m+4,12m-5\}&=\gcd\{12m+4,12m+1\}=\gcd\{12m+4,12m+7\}\\
&=\gcd\{12m+3,12m-1\}=\gcd\{12m+3,12m+2\}\\
&=\gcd\{12m+7,12m+5\}=\gcd\{12m+1,12m-3\}\\
&=\gcd\{12m+1,12m+2\}=\gcd\{12m+2,12m+5\}=1.
\end{align*}
The independence number when $n=5m+3$ is $\alpha(GP(5m+3,2))=\left\lfloor\frac{4(5m+3)}{5}\right\rfloor=4m+2$, implying our use of two even labels on the final six vertices is the maximum allowable. Therefore, $\pr(GP(n,3))=12m+7$ when $n=5m+3$.

Finally, when $n=5m+4$ we need to consider three cases when labeling the final eight vertices.  In each case, the labels on $u_{5m+3}$, $u_{5m+4}$, and $v_{5m+4}$ trivially have no common factors with their respective adjacent vertices $u_1$, $u_2$, and $v_1$.  
First, assume $5\nmid 12m+2$ and label the remaining vertices as
\begin{align}\label{lastcase}
    \begin{array}{cccccc}
         \ell(u_{5m+1})=12m+4, && \ell(u_{5m+2})=12m+8, \\
          \ell(u_{5m+3})=12m+3, && \ell(u_{5m+4})=12m+9, \\
         \ell(v_{5m+1})=12m+1, && \ell(v_{5m+2})=12m+5,  \\
         \ell(v_{5m+3})=12m+2, && \ell(v_{5m+4})=12m+7.
    \end{array}
\end{align}
As before, $v_{5m}$, $u_{5m}$, and $u_{5m-1}$ are all constructed the same as in Equation~\eqref{initiallabel}. 
Thus,
\begin{align*}
\gcd\{12m+4,12m-5\} &= \gcd\{12m+4,12m+3\}=\gcd\{12m+4,12m+1\}\\
&= \gcd\{12m+8,12m-1\}=\gcd\{12m+8,12m+5\}\\
&=\gcd\{12m+8,12m+9\} = \gcd\{12m+3,12m+2\}\\
&=\gcd\{12m+9,12m+7\}=\gcd\{12m+1,12m-3\} \\
&=\gcd\{12m+1,12m+5\}=\gcd\{12m+5,12m+2\}=1.
\end{align*}
Additionally, our final pair satisfies $\gcd\{12m+2,12m+7\}=1$ by our assumption of $5\nmid 12m+2$.

Next we assume $5\mid 12m+2$ and $7\nmid 12m+2$.  We label the eight vertices in the final block as in Equation~\eqref{lastcase} except reassign $\ell(u_{5m+1})=12m+2$ and $\ell(v_{5m+3})=12m+4$.  The following show these two vertices have labels that are relatively prime with their neighbors, where the assumption of $7\nmid 12m+2$ is necessary for the first $\gcd$ calculation:
\begin{align*}
u_{5m+1}&: \gcd\{12m+2,12m-5\}=\gcd\{12m+2,12m+3\}=\gcd\{12m+2,12m+1\}=1\\
v_{5m+3}&: \gcd\{12m+4,12m+5\}=\gcd\{12m+4,12m+3\}=\gcd\{12m+4,12m+7\}=1.
\end{align*}

Finally, we assume $5\mid 12m+2$ and $7\mid 12m+2$.  We label the vertices as follows:
\begin{align*}
    \begin{array}{cccccccccc}
         \ell(u_{5m+1})=12m && \ell(u_{5m+2})=12m+2 && \ell(u_{5m+3})=12m+7 && \ell(u_{5m+4})=12m+1,\\
         \ell(v_{5m+1})=12m+5 && \ell(v_{5m+2})=12m+3 && \ell(v_{5m+3})=12m+4 && \ell(v_{5m+4})=12m+9. \\
    \end{array}
\end{align*}
Since $12m+2$ is divisible by $5$ and $7$, we know that $12m$ is not divisible by $5$ or $7$ and likewise $12m+4$ is not divisible by $5$, resulting in 
$$\gcd\{12m,12m-5\}=\gcd\{12m,12m+5\}=\gcd\{12m,12m+7\}=\gcd\{12m+4,12m+9\}=1.$$  
The remaining adjacent pairs have relatively prime labels based on the following
\begin{align*}
\gcd\{12m+2,12m-1\} &=\gcd\{12m+2,12m+1\}=\gcd\{12m+2,12m+3\}\\
 &= \gcd\{12m+7,12m+4\}= \gcd\{12m+1,12m+9\}\\
&= \gcd\{12m+5,12m-3\}=\gcd\{12m+5,12m+3\}\\
&= \gcd\{12m+3,12m+4\}=1.
\end{align*}
In all three cases for the final eight vertices, our relatively prime condition is true while using three even labels in this block.  Since the independence number when $n=5m+4$ is $\alpha(GP(5m+4,2))=\left\lfloor\frac{4(5m+4)}{5}\right\rfloor=4m+3$, we have shown $\pr(GP(n,2))=12m+9$ when $n=5m+4$, concluding our fifth and final case of $n\pmod{5}$.
\end{proof}

We conclude this section with conjectures for the minimum coprime number of $GP(n,k)$ for larger cases of $k$. The independence numbers, as given in~\cite{BHM} and~\cite{FGS}, for the case of $k=3$ when $n$ is odd is  $\alpha(GP(n,3))=n-2$, and when $n=3k$ is $\alpha(GP(3k,k))=\left\lceil \frac{5k-2}{2}\right\rceil$.  These values lead to the following conjectures, which we have  verified for small values of $n$, although a minimum coprime labeling in each general case still alludes us.

\begin{conjecture}
For $n \geq 7$, the minimum coprime number for $GP(n,3)$ if $n$ is odd is
$$\pr(GP(n,3))=2n+3.$$
\end{conjecture}

\begin{conjecture}
For $k\geq 2$, the minimum coprime number for $GP(3k,k)$ is given by
$$\pr(GP(3k,k))=
\begin{cases}
7k & \text{if } k \text{ is odd} \\
7k+1 & \text{if } k \text{ is even}.\\
\end{cases}$$
\end{conjecture}

\section{Stacked Prisms}

We next turn our focus to the class of graphs known as the \textit{stacked prism}, also known as the generalized prism graph.  A stacked prism is defined as $Y_{m,n}=C_m\square P_n$ for $m\geq 3$ and $n\geq 1$. See Figure~\ref{fig:y36} for an example of $Y_{3,6}$ with a minimum coprime labeling. 

\begin{figure}[htb]
    \centering
    \includegraphics{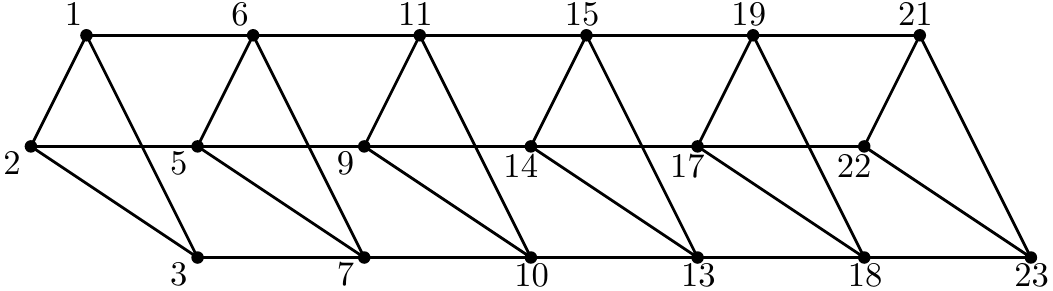}
    \caption{Example of a minimum coprime labeling of $Y_{3,6}$}
    \label{fig:y36}
\end{figure}

We first focus on the stacked triangular prism, $Y_{3,n}$, which has $3n$ vertices. Its independence number is $n$ since an independent set can contain at most one vertex from each triangle, and it is trivial to find such a set of $n$ vertices.  We demonstrate in the following result a way to apply a minimum coprime labeling based on how this independence number limits our use of even labels.

\begin{theorem}
The minimum coprime number for the stacked triangular prism is given by
$$\pr(Y_{3,n})=4n-1.$$
\end{theorem}
\begin{proof}
We refer to the vertices of $Y_{3,n}$ as $v_{i,j}$ where $i=1,\ldots, n$ and $j=1,2,3$.
Then the edges of the graph are of the form $v_{i,j}\sim v_{i+1,j}$ and $v_{i,j}\sim v_{i,k}$ for $j\neq k$.  We form a coprime labeling $\ell$ recursively by labeling the vertices on the $(i+1)^{\rm st}$ triangle with $\ell(v_{i+1,r}), \ell(v_{i+1,s}), \ell(v_{i+1,t})$ based on the labels chosen for the $i^{\rm th}$ triangle, $\ell(v_{i,r}), \ell(v_{i,s}), \ell(v_{i,t})$.  
First assign the labels $\ell(v_{1,1})=1$, $\ell(v_{1,2})=2$, and $\ell(v_{1,3})=3$.  Each subsequent $(i+1)^{\rm st}$ triangle for $1\leq i\leq n-1$ will use the labels $4i+1$, $4i+2$, and $4i+3$ in some order (or $4i$, $4i+1$, and $4i+3$ in one case) depending on which labels from the $i$th triangle are multiples of 3 and/or 5.  In each case we assume $\ell(v_{i,r})=4i-3$, $\ell(v_{i,s})=4i-2$, and $\ell(v_{i,t})=4i-1$.

\noindent \textbf{Case 1:}  Suppose $5\nmid 4i-3$ and $3\nmid 4i-2$.  In this case, we assign $\ell(v_{i+1,r})=4i+2$, $\ell(v_{i+1,s})=4i+1$, and $\ell(v_{i+1,t})=4i+3$.  The three edges within the $(i+1)^{\rm st}$ triangle have relatively prime labels since they are either consecutive integers or consecutive odd integers.  Since we assumed $5\nmid 4i-3$ and $3\nmid 4i-2$, the edges between the $i^{\rm th}$ and $(i+1)^{\rm st}$ triangle satisfy the following:
$$\gcd\{4i-3,4i+2\}=\gcd\{4i-2,4i+1\}=\gcd\{4i-1,4i+3\}=1.$$

\noindent \textbf{Case 2:} Assume $5\mid 4i-3$ and $3\nmid 4i-1$, which implies $5\nmid 4i-2$.  We then assign $\ell(v_{i+1,r})=4i+1$, $\ell(v_{i+1,s})=4i+3$, and $\ell(v_{i+1,t})=4i+2$.  Again the vertices within the newly labeled triangle have pairs of relatively prime labels.  Since $5\nmid 4i-2$ and $3\nmid 4i-1$, the edges between the two triangles satisfy the following:
$$\gcd\{4i-3,4i+1\}=\gcd\{4i-2,4i+3\}=\gcd\{4i-1,4i+2\}=1.$$

\noindent \textbf{Case 3:} Next suppose $5\nmid 4i-3$ and $5\nmid 4i-2$, in which we assign $\ell(v_{i+1,r})=4i+2$, $\ell(v_{i+1,s})=4i+3$, and $\ell(v_{i+1,t})=4i+1$.  As in previous cases, we only need to verify the edges between the $i^{\rm th}$ and $(i+1)^{\rm st}$ triangles have relatively prime labels on their endpoints, which is satisfied since our assumptions of $5\nmid 4i-3$ and $5\nmid 4i-2$ result in
$$\gcd\{4i-3,4i+2\}=\gcd\{4i-2,4i+3\}=\gcd\{4i-1,4i+1\}=1.$$

\noindent \textbf{Case 4:} There are two remaining other possible assumptions that can be made about factors of 3 or 5 that would allow all cases to be covered: either $5\mid 4i-3$ and $3\mid 4i-1$, or $5\nmid 4i-3$ and both $3$ and $5$ divide $4i-2$. We handle these by combining them into one final case. This is because if $5\mid 4i-3$ and $3\mid 4i-1$, then this implies $5\nmid 4i+1$ and $3,5\mid 4i+2$.  Therefore, having the first of these possible cases implies the second occurs on the next triangle, so we assign labels to the $(i+1)^{\rm st}$ and $(i+2)^{\rm nd}$ triangles at once while assuming $5\mid 4i-3$ and $3\mid 4i-1$.  We set $\ell(v_{i+1,r})=4i$, $\ell(v_{i+1,s})=4i+1$, and $\ell(v_{i+1,t})=4i+3$, as well as $\ell(v_{i+2,r})=4i+5$, $\ell(v_{i+2,s})=4i+6$, and $\ell(v_{i+2,t})=4i+7$.  Note that the labels on one of the edges of the $(i+1)^{\rm st}$ do not trivially satisfy the relatively prime condition.  However, $\gcd\{4i,4i+3\}=1$ since the assumption of $3\mid 4i-1$ implies $3\nmid 4i$, so the condition is satisfied nonetheless.  Also note that while the $(i+1)^{\rm st}$ triangle does not use the three consecutive labels that have been used in other cases, the $(i+2)^{\rm nd}$ triangle does use the three consecutive values that allow us to continue with our recursion to find the $(i+3)^{\rm rd}$ triangle next.

We now verify the six edges from the $i^{\rm th}$ triangle to the $(i+1)^{\rm st}$ and $(i+1)^{\rm st}$ triangle to the $(i+2)^{\rm nd}$ have relatively prime endpoints. It is clear that $\gcd\{4i-1, 4i+3\}=\gcd\{4i+3, 4i+7\}=1$.  Since we assumed $3\mid 4i-1$, we know $3\nmid 4i-2$ or $4i$, hence $\gcd\{4i-2,4i+1\}=\gcd\{4i-3,4i\}=1$.  Our assumption of $5\mid 4i-3$ implies $5\nmid 4i$ or $4i+1$, thus resulting in $\gcd\{4i,4i+5\}=\gcd\{4i+1,4i+6\}=1$. 

In all four cases, we have verified that each pair of adjacent vertices is labeled by relatively prime integers, thus resulting in a coprime labeling once our recursion approach reaches $i=n-1$.  The final triangle will use $4(n-1)+3=4n-1$ as its largest label.  Thus, since we used the maximum amount of $n$ even labels based on the independence number and the smallest possible odd labels, we have proven $\pr(Y_{3,n})=4n-1$.
\end{proof}

We next investigate a minimum coprime labeling of the stacked pentagonal prism, $Y_{5,n}$, which has~$5n$ vertices.  Similarly to $Y_{3,n}$, its independence number is determined based on at most $2$ independent vertices being on each pentagon.  Thus, the independence number is $\alpha(Y_{5,n})=2n$, which leads to the following result on the minimum coprime number of the stacked pentagonal prism. See Figure~\ref{fig:y56} for an example of $Y_{5,6}$ with a minimum coprime labeling. 

\begin{figure}[htb]
    \centering
    \includegraphics{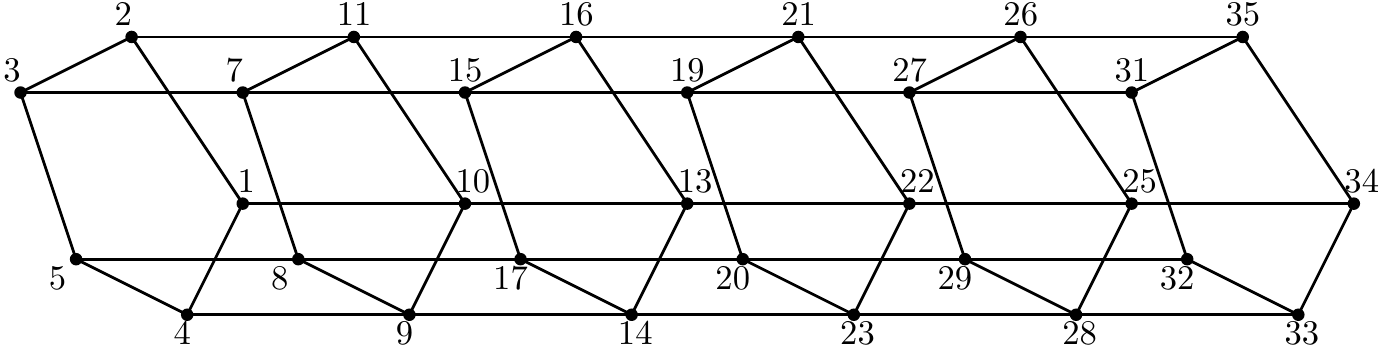}
    \caption{Example of a minimum coprime labeling of $Y_{5,6}$}
    \label{fig:y56}
\end{figure}

\begin{theorem}
The minimum coprime number for the stacked pentagonal prism graph is
$$\pr(Y_{5,n})=6n-1.$$
\end{theorem}
\begin{proof}
We will call $v_{i,j}$ the vertices of $Y_{5,n}$ where $i=1,\ldots, n$ and $j=1,2,3,4,5$.  We initially assign a labeling $\ell$ for $i=1,\ldots, 70$ as follows if $i$ is odd
$$\ell(v_{i,1})=6i-3,\; \ell(v_{i,2})=6i-4,\; \ell(v_{i,3})=6i-5,\; \ell(v_{i,4})=6i-2,\; \ell(v_{i,5})=6i-1,$$
and using the following if $i$ is even
$$\ell(v_{i,1})=6i-5,\;\ell(v_{i,2})=6i-1,\;\ell(v_{i,3})=6i-2, \;\ell(v_{i,4})=6i-3,\;\ell(v_{i,5})=6i-4.$$
Since the labeling differs on odd- and even-indexed pentagons, there are twenty types of adjacent pairs to consider as having relatively labels. If $i$ is odd and $k$ is even, the edges $v_{i,1}v_{i,2}$, $v_{i,2}v_{i,3}$, $v_{i,4}v_{i,5}$, $v_{k,2}v_{k,3}$, $v_{k,3}v_{k,4}$, $v_{k,4}v_{k,5}$, and $v_{k,5}v_{k,1}$ are labeled by consecutive integers.  The edges $v_{i,5}v_{i,1}$, $v_{i,1}v_{i+1,1}$, $v_{k,1}v_{k,2}$, and $v_{k,1}v_{k+1,1}$ are labeled by odd integers that have a difference of $2$, $4$, or $8$.  Edges of the form $v_{i,3}v_{i,4}$, $v_{i,2}v_{i+1,2}$, $v_{i,3}v_{i+1,3}$, $v_{i,5}v_{i+1,5}$, $v_{k,2}v_{k+1,2}$, $v_{k,3}v_{k+1,3}$, or $v_{k,5}v_{k+1,5}$ have labels that differ by 3 or 9, in which these labels are not multiples of 3.  At this point, we have shown that the vertices on eighteen edges are relatively prime. 

The edges $v_{i,4}v_{i+1,4}$ and $v_{k,4}v_{k+1,4}$ have labels that differ by 5 and 7, respectively, and hence may not be relatively prime for certain $i$ and $k$.  Rather than detail numerous cases of alterations that need to be made to the labeling to fix these, we instead list the reassigned labels of 46 vertices from the 350 labels in which $i=1,\ldots, 70$, as seen in Table~\ref{first350}.  The labels in bold are the 46 that were reassigned to avoid adjacent labels sharing multiples of 5 and 7.

\begin{table}[htb]
    \centering
    \begin{tabular}{|c|c|c|c|c|}
    \hline
3	&	2	&	1	&	4	&	5\\
\hline
7	&	11	&	10	&	9	&	8\\
\hline
15	&	\textbf{16}	&	13	&	\textbf{14}	&	17\\
\hline
19	&	\textbf{21}	&	22	&	\textbf{23}	&	20\\
\hline
27	&	26	&	25	&	28	&	29\\
\hline
31	&	35	&	34	&	33	&	32\\
\hline
39	&	38	&	37	&	40	&	41\\
\hline
43	&	\textbf{45}	&	46	&	\textbf{47}	&	44\\
\hline
51	&	\textbf{52}	&	49	&	\textbf{50}	&	53\\
\hline
55	&	59	&	58	&	57	&	56\\
\hline
63	&	62	&	61	&	64	&	65\\
\hline
67	&	71	&	70	&	69	&	68\\
\hline
75	&	74	&	73	&	76	&	77\\
\hline
79	&	83	&	82	&	81	&	80\\
\hline
87	&	86	&	85	&	88	&	89\\
\hline
91	&	95	&	94	&	93	&	92\\
\hline
\textbf{101}	&	98	&	97	&	100	&	\textbf{99}\\
\hline
\textbf{105}	&	107	&	106	&	\textbf{103}	&	104\\
\hline
\textbf{113}	&	110	&	109	&	112	&	\textbf{111}\\
\hline
115	&	119	&	118	&	117	&	116\\
\hline
123	&	122	&	121	&	124	&	125\\
\hline
127	&	131	&	130	&	129	&	128\\
\hline
135	&	134	&	133	&	136	&	137\\
\hline
    \end{tabular}
    \hfill
    \begin{tabular}{|c|c|c|c|c|}
    \hline   
139	&	143	&	142	&	141	&	140\\
\hline
147	&	146	&	145	&	148	&	149\\
\hline
151	&	\textbf{153}	&	154	&	\textbf{155}	&	152\\
\hline
159	&	\textbf{160}	&	157	&	\textbf{158}	&	161\\
\hline
163	&	167	&	166	&	165	&	164\\
\hline
171	&	170	&	169	&	172	&	173\\
\hline
175	&	179	&	178	&	177	&	176\\
\hline
183	&	\textbf{184}	&	181	&	\textbf{182}	&	185\\
\hline
187	&	\textbf{189}	&	190	&	\textbf{191}	&	188\\
\hline
195	&	194	&	193	&	196	&	197\\
\hline
199	&	203	&	202	&	201	&	200\\
\hline
207	&	206	&	205	&	208	&	209\\
\hline
211	&	\textbf{213}	&	214	&	\textbf{215}	&	212\\
\hline
219	&	\textbf{220}	&	217	&	\textbf{218}	&	221\\
\hline
223	&	227	&	226	&	225	&	224\\
\hline
231	&	230	&	229	&	232	&	233\\
\hline
235	&	239	&	238	&	237	&	236\\
\hline
243	&	242	&	241	&	244	&	245\\
\hline
247	&	251	&	250	&	249	&	248\\
\hline
255	&	254	&	253	&	256	&	257\\
\hline
259	&	263	&	262	&	261	&	260\\
\hline
267	&	\textbf{268}	&	265	&	\textbf{266}	&	269\\
\hline
\textbf{275}	&	\textbf{273}	&	274	&	\textbf{271}	&	272\\
\hline
    \end{tabular}
    \hfill
    \begin{tabular}{|c|c|c|c|c|}
    \hline   
279	&	278	&	277	&	280	&	281\\
\hline
\textbf{287}	&	\textbf{285}	&	286	&	\textbf{283}	&	284\\
\hline
291	&	\textbf{292}	&	289	&	\textbf{290}	&	293\\
\hline
295	&	299	&	298	&	297	&	296\\
\hline
303	&	302	&	301	&	304	&	305\\
\hline
307	&	311	&	310	&	309	&	308\\
\hline
315	&	314	&	313	&	316	&	317\\
\hline
319	&	323	&	322	&	321	&	320\\
\hline
327	&	326	&	325	&	328	&	329\\
\hline
331	&	335	&	334	&	333	&	332\\
\hline
339	&	338	&	337	&	340	&	341\\
\hline
343	&	\textbf{345}	&	346	&	\textbf{347}	&	344\\
\hline
351	&	\textbf{352}	&	349	&	\textbf{350}	&	353\\
\hline
355	&	\textbf{357}	&	358	&	\textbf{359}	&	356\\
\hline
363	&	362	&	361	&	364	&	365\\
\hline
367	&	371	&	370	&	369	&	368\\
\hline
375	&	374	&	373	&	376	&	377\\
\hline
379	&	383	&	382	&	381	&	380\\
\hline
387	&	386	&	385	&	388	&	389\\
\hline
391	&	395	&	394	&	393	&	392\\
\hline
399	&	398	&	397	&	400	&	401\\
\hline
403	&	\textbf{405}	&	406	&	\textbf{407}	&	404\\
\hline
411	&	\textbf{412}	&	409	&	\textbf{410}	&	413\\
\hline
415	&	419	&	418	&	417	&	416\\
\hline
    \end{tabular}
    \caption{Labeling of the Stacked Pentagon for $i=1$ to $70$}
    \label{first350}
\end{table}

Two important facts regarding these newly assigned labels can be observed: each label is relatively prime with any adjacent label, resulting in a coprime labeling of the graph up to $n=70$, and also that the largest difference between adjacent labels is 10.

To label the stacked pentagonal prism graph when $n> 70$, we assign for $i>70$ and $j=1,\ldots 5$ the label $\ell(v_{i,j})=\ell(v_{i-70,j})+420$.  Since the greatest distance between adjacent labels on the first~$70$ pentagons was $10$, only common prime factors $2$, $3$, $5$, and $7$ need to be considered.  The shift by $420$, which only has these four prime numbers as factors, results in the relatively prime condition remaining satisfied for all vertex pairs with $i\geq 71$.  The only exceptions that need to be verified are edges of the form $v_{70m,j}v_{70m+1,j}$ for some positive integer $m$.  We have the following pairs of labels that are adjacent:
\begin{align*}
    \{\ell(v_{70m,1})=420m-5&,\; \ell(v_{70m+1,1})=420m+3\}\\
    \{\ell(v_{70m,2})=420m-1&,\; \ell(v_{70m+1,2})=420m+2\}\\
    \{\ell(v_{70m,3})=420m-2&,\; \ell(v_{70m+1,3})=420m+1\}\\
    \{\ell(v_{70m,4})=420m-3&,\; \ell(v_{70m+1,4})=420m+4\}\\
    \{\ell(v_{70m,5})=420m-4&,\; \ell(v_{70m+1,5})=420m+5.\}
\end{align*}
The first pair is relatively prime since they are odd integers differing by a power of $2$.  The second, third, and fifth pairs are not multiples of 3 and are separated by either 3 or 9.  Finally, the fourth pair are separated by 7, but since 7 divides 420, neither of these are multiples of 7.  Thus, these adjacent pairs of vertices all have relatively prime labels, making this a coprime labeling for $n\geq 71$ as well.

Note that the largest label $\ell(v_{i,j})$ for $j=1,\ldots, 5$ is $6i-1$ for all $i$, hence the largest label for $Y_{5,n}$ is $6n-1$.  Since we used the maximum of $2$ even labels on each pentagon along with the smallest possible odd labels, we have proven that $\pr(Y_{5,n})=6n-1$.
\end{proof}

For the graphs $Y_{3,n}$ and $Y_{5,n}$, the minimum coprime numbers were directly correlated to their independence number.  This number is easy to obtain for stacked prisms involving larger odd cycles as well.  Since $\alpha(C_{2k+1})=k$, one can observe that $\alpha(Y_{2k+1,n})=kn$.  A minimum coprime labeling could use at most $kn$ even labels, forcing $(2k+1)n-kn=(k+1)n$ odd labels to be used, leading to the following conjecture. 

\begin{conjecture}
The minimum coprime number for stacked $(2k+1)$-gon prism graph is
$$\pr(Y_{2k+1,n})=2(k+1)n-1.$$
\end{conjecture}

It should be noted to close this section that while we focused on the odd case, the stacked prisms with even-length cycles, $Y_{2k,n}$, satisfies $\alpha(Y_{2k,n})=kn=\frac{|V(Y_{2k,n})|}{2}$.  Therefore, one would hope a prime labeling exists in this case, but this remains an open problem. 

\section{Variation of Generalized Petersen Graph}


We extend the definition of the generalized Petersen graph for the case of even $n$ and $k=\frac{n}{2}$.  We denote this graph by $GP^*(2k,k)$ and use the same notation for the vertices $u_1,\ldots, u_n$ and $v_1,\ldots, v_n$ with edges $v_i\sim v_{i+1}$, $v_i\sim u_i$ and $u_i\sim u_{i+k}$ where indices are calculated modulo $n$.  Note that this graph differs from generalized Petersen graphs since $deg(u_i)=2$ for all $i=1,\ldots, n$ instead of the usual degree of $3$ for $GP(n,k)$. For an example of a minimum coprime labeling of $GP^*(20,10)$, see Figure~\ref{fig:gp2kk}.

\begin{figure}[htb]
    \centering
    \includegraphics{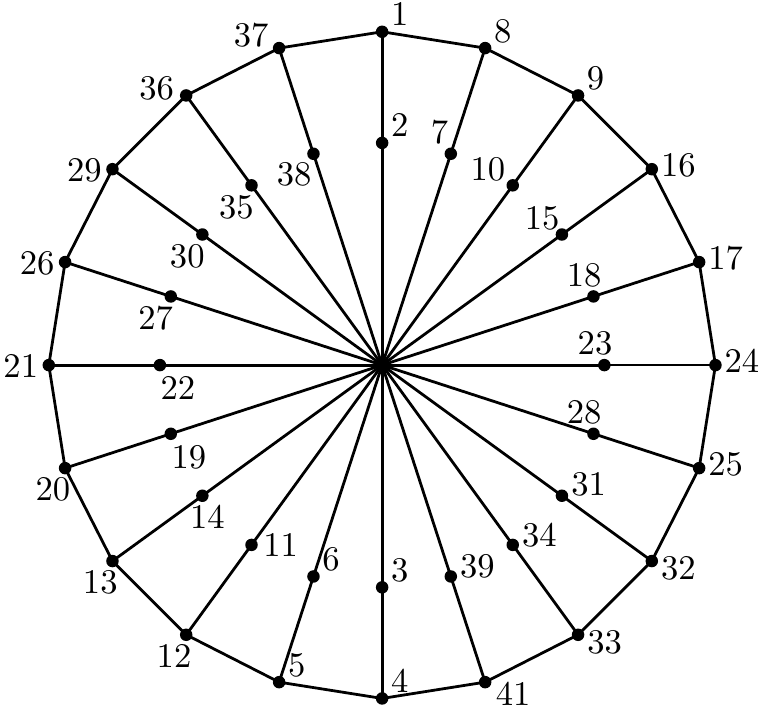}
    \caption{Example of a minimum coprime labeling of $GP^*(20,10)$.}
    \label{fig:gp2kk}
\end{figure}

\begin{theorem}
For the graph $GP^*(2k,k)$ in which $k\geq 2$, we have the following
\begin{enumerate}
\item $GP^*(2k,k)$ is prime if $k$ is odd,
\item $\pr(GP^*(2k,k))=4k+1$ if $k$ is even.
\end{enumerate}
\end{theorem}

\begin{proof}
We first assume $k$ is odd.  We begin by labeling the vertices of $GP^*(2k,k)$ as follows

\begin{equation}\label{variation}
\ell(v_i)=\begin{cases}
4i-3 & \mbox{ for $i=1,3,5,\ldots, k$}\\
4i & \mbox{ for $i=2,4,6,\ldots, k-1$}\\
4i-4k &\mbox{ for $i=k+1,k+3,k+5,\ldots, 2k$}\\
4i-3-4k &\text{ for $i=k+2,k+4,k+6,\ldots, 2k-1$.}\\
\end{cases}
\end{equation}

\begin{equation*}
\ell(u_i)=\begin{cases}
4i-2 & \mbox{ for $i=1,3,5,\ldots, k$}\\
4i-1 & \mbox{ for $i=2,4,6,\ldots, k-1$}\\
4i-1-4k &\mbox{ for $i=k+1,k+3,k+5,\ldots, 2k$}\\
4i-2-4k &\text{ for $i=k+2,k+4,k+6,\ldots, 2k-1$.}\\
\end{cases}
\end{equation*}

For the pairs $v_i\sim v_{i+1}$ and $v_{i+k}\sim v_{i+1+k}$ for $i\in\{1,\ldots,k-1\}$, it is clear that $\gcd\{\ell(v_i),\ell(v_{i+1})\}\in\{1,7\}$ and $\gcd\{\ell(v_{i+k}),\ell(v_{i+k+1})\}\in \{1,7\}$. Notice that $\gcd\{\ell(v_{k}),\ell(v_{k+1})\}=\gcd\{4k-3,4\}=1$ and $\gcd\{\ell(v_{2k}),\ell(v_{1})\}=\gcd\{4k-1,1\}=1$. For the edges $u_i u_{i+k}$, it is easily verified that when $i\in\{1,3,\ldots,k\}$,
\[
\gcd\{\ell(u_i),\ell(u_{i+k})\}=\gcd\{4i-2,4(i+k)-1-4k\} = \gcd\{4i-2,4i-1\}=1.
\]
Similarly, $\gcd\{\ell(u_i),\ell(u_{i+k})\}=1$ when $i\in\{2,4,\ldots,k-1\}$ and $\gcd\{\ell(u_i),\ell(v_i)\} = 1$ for all $i=1,\ldots, 2k$ since these pairs are consecutive integers. Thus, our only concern with the labeling is when $u_i$ and $u_{i+1}$ are both divisible by 7. 
We handle these instances by breaking the proof into several cases based on the remainder of $\ell(v_i)/3$ and whether or not $5$ divides $\ell(v_i)+2$. For the sake of simplicity, let $\ell(v_i)=a$ where $i\in\{1,3,\ldots,k\}$. 
Notice that in our labeling, $\ell(v_i)=a$ is odd in these cases in which $\ell(v_{i+1})-\ell(v_i)=7$.
In each of the cases below, we are supposing that $a\equiv 0 \pmod{7}$. 

\noindent \textbf{Case 1:} Suppose that $a\equiv 1\pmod{3}$.

In this case, swap the labels $a$ and $a-2$. Notice that $\ell(u_{i-1})=a-2$. Since $a$ is not divisible by $3$, $\gcd\{a,a-3\}=1$. Since $a-2$ is also not divisible by 3, $\gcd\{a-2,a+7\}=\gcd\{a-2,a+1\}=1$. Thus, the two labels involved in the swap are relatively prime with all adjacent labels.\hfill $\diamond$

\noindent \textbf{Case 2:} Suppose that $a\equiv 0\pmod{3}$ and $5 \nmid a$.

In this case, swap the labels $a+5$ and $a+7$ on the vertices $v_{i+1}$ and $u_{i+1+k}$. Since $a$ is not divisible by $5$, $\gcd\{a,a+5\}=1$. Since both $a+5$ and $a+7$ are not divisible by $3$, $\gcd\{a+5,a+8\}=\gcd\{a+4,a+7\}=1$. 
Thus, all newly adjacent pairs of labels after making this swap are relatively prime.\hfill $\diamond$

\noindent \textbf{Case 3:} Suppose that $a\equiv 0\pmod{15}$.

The labels $a-1,a,a+1,a+5,a+6,a+7$ will become $a+7,a+6,a+5,a-1,a,a+1$, respectively. Based on the placement of the swapped labels, we only need to check the $\gcd$ between six newly adjacent pairs of labels. Since $a$ is a multiple of $3$, $5$, and $7$, we have 
\[
\gcd\{a+1,a+8\}=\gcd\{a-1,a+4\},\gcd\{a+5,a+2\}=\gcd\{a+7,a-2\}=\gcd\{a+1,a+6\}=1.
\]
Note that if $i=2k-1$, instead of the label $a+8$, we have label $1$ on $v_{i+2}$, which is still relatively prime with $a+1$.
Notice that the label $a+7$ is also now on a vertex adjacent to the vertex labeled $a-8$. Since $a\equiv 0\pmod{15}$, $\gcd\{a-8,a+7\}=1$. Therefore, the reassigned labels are relatively prime with any newly adjacent label.\hfill $\diamond$

\noindent \textbf{Case 4:} Suppose that $a\equiv 2\pmod{3}$ and $5 \nmid a+2$.

In this case, swap the labels $a$ and $a+2$. Since both $a$ and $a+2$ are not divisible by $3$ and $a+2$ is not divisible by $5$, it follows that
\[
\gcd\{a,a+3\} = \gcd\{a+2,a+7\} = \gcd\{a+2,a-1\}=1.
\]
Hence any new adjacencies after the swap consist of relatively prime labels. \hfill $\diamond$

\noindent \textbf{Case 5:} Suppose that $a\equiv 2\pmod{3}$ and $5\mid a+2$.

The labels $a,a+1,a+2,a+6,a+7,a+8$ are reassigned as $a+6,a+7,a+8,a+2,a+1,a$, respectively. Based on the placement of the swapped labels, as in Case 3, we need only check the $\gcd$ of six pairs of labels. Since $a\equiv 2\pmod{3}$, $a+2\equiv 0\pmod{5}$, and $a\equiv 0\pmod{7}$, 
\[
\gcd\{a-1,a+6\}=\gcd\{a+3,a+8\}=\gcd\{a+2,a+5\}=\gcd\{a+1,a+6\}=\gcd\{a,a+9\}=1.
\]
If $\ell(v_{i+2})=a+8$ in the original labeling of the vertices, then since $a$ is odd, $a\equiv 2\pmod{3}$, and $a+2\equiv 0\pmod{5}$, it follows that $\gcd\{a+15,a\}=1$. 
Since $k$ is odd, it is not possible for $\ell(v_{i+2})$ to be $4$.
It is possible that $\ell(v_{i+2})=1$ in the case of $i=2k-1$, but in this case, $\ell(v_{i+3})=8$ and the label~$a$ is adjacent to $8$, which again proves our claim.
Thus, all new pairs of adjacent labels are relatively prime.\hfill $\diamond$

Combining all cases, we have proven that the labeling is prime for all odd $k$.
Now suppose that $k$ is an even integer. In order to have $k$ even labels on the vertices $v_1,\ldots, v_n$, every other vertex $v_i$ must be even. However, for each $i$, both $v_i$ and $v_{i+k}$ will be labeled with an even integer and so neither $u_i$ or $u_{i+k}$ can be labeled with an even number. Thus, less than $k$ even labels can be used on the vertices $v_1,\ldots, v_n$.  We cannot use more than $k$ even labels on $u_1,\ldots, u_n$ since only one of $\ell(u_i),\ell(u_{i+k})$ can be even for each $i=1,\ldots, k$. Therefore, we have $\alpha(GP^*(2k,k))<2k=\frac{1}{2}|V|$, so by Lemma~\ref{ind}, this graph is not prime when $k$ is even.

We then create a minimum coprime labeling when $k$ is even by starting with the labeling $\ell$ as defined in Equation~\eqref{variation}, which originally assigned $\ell(v_{k})=4k$.  Then make the same alterations to $\ell$ as described in Cases 1-5.
If these left $\ell(v_k)$ unchanged, then reassign the label on $v_{k}$ to $4k+1$. Then $\gcd\{4k+1,4\}=\gcd\{4k-1,4k+1\}=\gcd\{4k-7,4k+1\}=1$. Now suppose that $\ell(v_{k})$ is relabeled by one of the cases above, which could have occurred in Case 2, 3, or 5. Assuming we are in Case 2, recall $\ell(u_{2k})=a+5$ and $\ell(v_k)=a+7=4k$ are the only two values swapped from Equation~\eqref{variation}.  
Instead, let $\ell(u_{2k})=4k$ and $\ell(v_k)=4k+1$. 
Then the lables of the vertices adjacent to $u_{2k}$ are $4k-3$ and $4k-1$, and the labels of the vertices adjacent to $v_k$ are $4$, $4k-7$, and $4k-1$. 
Since $4k+1$ is odd and $a=4k-7\equiv 0\pmod{3}$, it follows that $\ell(u_{2k})$ and $\ell(v_k)$ are relatively prime with the labels of the corresponding vertex neighbors. 
Suppose we are in Case 3 where $\ell(v_{k})$ is now swapped to be $4k-6$. Instead, we now relabel $v_{k}$ as $4k+1$. Then the labels of vertices adjacent to $v_{k}$ are labeled $4$, $4k-7$, and $4k-1$, all of which are relatively prime to $4k+1$.
Finally, suppose that we are in Case~5 where $\ell(v_{k})$ is currently $4k-6$ and $\ell(v_{k+1})$ is $4k-7$. We relabel $v_{k}$ as $4k+1$. 
The labels of the vertices adjacent to $v_{k}$ are $4k-5$, $4k-7$, and $4k-1$. Since $4k-5\equiv 1\pmod{3}$ based on assumptions of Case~5, it follows that $\gcd\{4k+1,4k-5\}=\gcd\{4k+1,4k-1\}=\gcd\{4k+1,4k-7\}=1$, and so each pair labels of adjacent vertices are relatively prime.  We have shown this is a coprime labeling for the case of even $k$ with largest label $4k+1$.  Therefore, since the graph isn't prime in this case, we have proven $\pr(GP^*(2k,k))=4k+1$.
\end{proof}

\bibliographystyle{amsplain}
\bibliography{references.bib}

\appendix

\section{Additional Generalized Petersen Results}
For Theorems~\ref{n+4}-\ref{n+6}, the proof for why the labeling is a coprime labeling is omitted as the pattern is clearly given in the tables and can be verified in each case. The techniques used are all outlined in Observation~\ref{easy}.  Recall that we are working under the assumption that $n$ is odd for the following theorems.

\begin{theorem}\label{n+4}
If $n+4$ is prime, then $\pr(GP(n,1))=2n+1$.
\end{theorem}

\begin{proof}
We can assume $n\not\equiv 2\pmod{3}$, else $n+4$ would be divisible by 3 and hence not prime.
If $n\equiv 0\pmod{3}$, then we use the labeling defined in Table~\ref{tab:n+4}. 
If $n\equiv 1\pmod{3}$, then we alter the previous labeling and instead use the labeling defined in Table~\ref{tab:n+4b}.

\end{proof}

\begin{theorem}\label{n-2}
If $n-2$ is prime and $n>5$, then $\pr(GP(n,1))=2n+1$.
\end{theorem}

\begin{proof}
We may assume $n\not\equiv 2\pmod{3}$ to avoid $n-2$ being a multiple of 3. If $n\equiv 0\pmod{3}$, then we use the labeling defined in Table~\ref{tab:n-2}.  If $n\equiv 1\pmod{3}$, we instead label the graph using the labeling in Table~\ref{tab:n-2b}.  Note that in both cases, the labels on most edges $u_i\sim v_i$ are relatively prime since the difference of the labels is the prime $n-2$, but a swap is necessary to avoid both labels being multiples of that prime number.
\end{proof}

\begin{theorem}\label{n-4}
If $n-4$ is prime and $n>7$, then $\pr(GP(n,1))=2n+1$.
\end{theorem}

\begin{proof}
We may assume $n\not\equiv 1\pmod{3}$, else $n-4$ is a multiple of 3.
If $n\equiv 0\pmod{3}$, then we use the labeling defined in Table~\ref{tab:n-4}. 
If $n\equiv 2\pmod{3}$, then use the same labeling defined in Table~\ref{tab:n-4b}. 
\end{proof}

\begin{theorem}\label{2n+3}
If $2n+3$ is prime, then $\pr(GP(n,1))=2n+1$. 
\end{theorem}

\begin{proof}
We assume $n\not\equiv 0\pmod{3}$ to avoid $2n+3$ being divisible by $3$.
If $n\equiv 1\pmod{3}$, then we use the labeling defined in Table~\ref{tab:2n+3}. 
If $n\equiv 2\pmod{3}$, then use the same labeling defined in Table~\ref{tab:2n+3b} except $n+1$ is labeled as $n-1$ instead. 

\end{proof}

\begin{theorem}\label{2n-3}
If $2n-3$ is prime, then $\pr(GP(n,1))=2n+1.$
\end{theorem}

\begin{proof}
We may assume $n\not\equiv 0\pmod{3}$, otherwise $2n-3$ is a multiple of $3$.
If $n\equiv 2\pmod{3}$, then we use the labeling defined in Table~\ref{tab:2n-3}. 
If $n\equiv 1\pmod{3}$, then we use the labeling defined in Table~\ref{tab:2n-3b}.

\end{proof}

\begin{theorem}\label{2n-5}
If $2n-5$ is prime, then $\pr(GP(n,1))=2n+1.$
\end{theorem}

\begin{proof}
We can assume $n\not\equiv 1\pmod{3}$ to avoid $3$ dividing into $2n-5$.
If $n\equiv 0$ or $2\pmod{3}$ and $n\equiv 1,2,$ or $4\pmod{5}$, then we use the labeling defined in Table~\ref{tab:2n-5}. 
When $n\equiv 0$ or $2\pmod{3}$ and $n\equiv 0\pmod{5}$, then $2n-5\equiv 0\pmod{5}$, so we ignore these cases.
When $n\equiv 0\pmod{3}$ and $n\equiv 3\pmod{5}$, then we use the labeling defined in Table~\ref{tab:2n-5b}. 
When $n\equiv 2\pmod{3}$ and $n\equiv 3\pmod{5}$, then we use the labeling defined in Table~\ref{tab:2n-5c}. 

\end{proof}

\begin{theorem}\label{n+6}
If $n+6$ is prime, then $\pr(GP(n,1))=2n+1.$
\end{theorem}

\begin{proof}
Assume that $n\not\equiv 0\pmod{3}$, else $n+6$ is divisible by 3.
If $n\equiv 1$ or $2\pmod{3}$ and $n\equiv 0,2,$ or $3\pmod{5}$, then we use the labeling defined in Table~\ref{tab:n+6}. 
If $n\equiv 2\pmod{3}$ and $n\equiv 1\pmod{5}$, then we use the labeling defined in Table~\ref{tab:n+6b}. 
If $n\equiv 1 \pmod{3}$ and $n\equiv 1\pmod{5}$, then we use the labeling defined in Table~\ref{tab:n+6c}.
If $n\equiv 1$ or $2\pmod{3}$ and $n\equiv 4\pmod{5}$, then $n+6$ is divisible by 5, so these cases are removed.


\end{proof}

\begin{table}[htb]
    \centering
    \begin{tabular}{|c|c|c|c|c|c|c|}
    \hline
         1& 2 & $\cdots$ & $n-3$ & $n-2$ & $n-1$ & $n$   \\
    \hline
         $n+5$ & $n+6$ & $\cdots$  & $2n+1$ & $n+1$ & $n+2$ & $n+4$   \\
    \hline
    \end{tabular}
    \caption{Labeling for Theorem~\ref{n+4} when $n\equiv 0\pmod{3}$}
    \label{tab:n+4}

\vspace{.2in}

    \begin{tabular}{|c|c|c|c|c|c|c|}
    \hline
         1& 2 & $\cdots$ & $n-3$ & $n-2$ & $n+2$ & $n+1$   \\
    \hline
         $n+5$ & $n+6$ & $\cdots$  & $2n+1$ & $n$ & $n+3$ & $n+4$   \\
    \hline
    \end{tabular}
    \caption{Labeling for Theorem~\ref{n+4} when $n\equiv 1\pmod{3}$}
    \label{tab:n+4b}

\vspace{.2in}

    \begin{tabular}{|c|c|c|c|c|c|c|c|c|}
    \hline
         1& 2 & 3 & 4 & $\cdots$ & $n-3$ & $n-2$ & $n-1$ & $n$   \\
    \hline
         $2n-1$ & $2n+1$ & $n+1$ & $n+2$ & $\cdots$ & $2n-5$ & $2n-2$  & $2n-3$ & $2n-4$   \\
    \hline
    \end{tabular}
    \caption{Labeling for Theorem~\ref{n-2} when $n\equiv 0\pmod{3}$}
    \label{tab:n-2}

\vspace{.2in}

    \begin{tabular}{|c|c|c|c|c|c|c|c|c|c|c|}
    \hline
         1& 2 & 3 & 4 & $\cdots$ & $n-5$ & $n-4$ & $n-3$ & $n-2$ & $n-1$ & $n$   \\
    \hline
         $2n-1$ & $2n+1$ & $n+1$ & $n+2$ & $\cdots$ & $2n-7$ & $2n-4$ & $2n-5$ & $2n-6$  & $2n-3$ & $2n-2$   \\
    \hline
    \end{tabular}
    \caption{Labeling for Theorem~\ref{n-2} when $n\equiv 1\pmod{3}$}
    \label{tab:n-2b}
\end{table}

\clearpage

\begin{table}[htb]
\resizebox{\textwidth}{!}{%
    \centering
     \begin{tabular}{|c|c|c|c|c|c|c|c|c|c|c|c|c|c|c|}
    \hline
         1& 2 & 3 & 4 & 5 & 6 & $\cdots$ & $n-7$ & $n-6$ & $n-5$ & $n-4$ & $n-3$ & $n-2$ & $n-1$ & $n$   \\
    \hline
         $2n-3$ & $2n-1$ & $2n-2$ & $2n+1$ & $n+1$ & $n+2$ & $\cdots$ & $2n-11$ & $2n-8$ & $2n-9$ & $2n-10$ & $2n-7$ & $2n-6$  & $2n-5$ & $2n-4$   \\
    \hline
    \end{tabular}}
    \caption{Labeling for Theorem~\ref{n-4} when $n\equiv 0\pmod{3}$}
    \label{tab:n-4}
\vspace{.2in}
    \resizebox{\textwidth}{!}{%
    \begin{tabular}{|c|c|c|c|c|c|c|c|c|c|c|c|c|}
    \hline
         1& 2 & 3 & 4 & 5 & 6 & $\cdots$ & $n-5$ & $n-4$ & $n-3$ & $n-2$ & $n-1$ & $n$   \\
    \hline
         $2n-3$ & $2n-1$ & $2n$ & $2n+1$ & $n+1$ & $n+2$ & $\cdots$ & $2n-9$ & $2n-6$ & $2n-7$ & $2n-8$  & $2n-5$ & $2n-4$   \\
    \hline
    \end{tabular}}
    \caption{Labeling for Theorem~\ref{n-4} when $n\equiv 2\pmod{3}$}
    \label{tab:n-4b}
\vspace{.2in}
    \centering
    \begin{tabular}{|c|c|c|c|c|c|c|c|}
    \hline
         1& 2 & 3 & 4  & $\cdots$ & $n-2$ & $n+1$ & $n+2$  \\
    \hline
         $n$ & $2n+1$ & $2n$ & $2n-1$ & $\cdots$  & $n+5$ & $n+4$ & $n+3$ \\
    \hline
    \end{tabular}
    \caption{Labeling for Theorem~\ref{2n+3} when $n\equiv 1\pmod{3}$}
    \label{tab:2n+3}
\vspace{.2in}
    \centering
    \begin{tabular}{|c|c|c|c|c|c|c|c|}
    \hline
         1& 2 & 3 & 4  & $\cdots$ & $n-2$ & $n$ & $n+3$  \\
    \hline
         $n-1$ & $2n+1$ & $2n$ & $2n-1$ & $\cdots$  & $n+5$ & $n+4$ & $n+2$ \\
    \hline
    \end{tabular}
    \caption{Labeling for Theorem~\ref{2n+3} when $n\equiv 2\pmod{3}$}
    \label{tab:2n+3b}
\vspace{.2in}
    \centering
    \begin{tabular}{|c|c|c|c|c|c|c|c|}
    \hline
         1& 2 & 3 & 4  & $\cdots$ & $n-2$ & $2n-3$ & $2n$  \\
    \hline
         $2n-2$ & $2n-5$ & $2n-6$ & $2n-7$ & $\cdots$  & $n-1$ & $2n-1$ & $2n+1$ \\
    \hline
    \end{tabular}
    \caption{Labeling for Theorem~\ref{2n-3} when $n\equiv 2\pmod{3}$}
    \label{tab:2n-3}
\vspace{.2in}
    \centering
    \begin{tabular}{|c|c|c|c|c|c|c|c|}
    \hline
         1& 2 & 3 & 4  & $\cdots$ & $n-2$ & $2n$ & $2n+1$  \\
    \hline
         $2n-4$ & $2n-5$ & $2n-6$ & $2n-7$ & $\cdots$  & $n-1$ & $2n-3$ & $2n-1$ \\
    \hline
    \end{tabular}
    \caption{Labeling for Theorem~\ref{2n-3} when $n\equiv 1\pmod{3}$}
    \label{tab:2n-3b}

%
\vspace{.2in}
    \centering
    \begin{tabular}{|c|c|c|c|c|c|c|c|}
    \hline
         1& 2 & 3 & $\cdots$ & $n-3$ & $2n-5$ & $2n+1$ & $2n$  \\
    \hline
         $2n-6$ & $2n-7$ & $2n-8$ & $\cdots$ & $n-2$ & $2n-3$ & $2n-2$ & $2n-1$ \\
    \hline
    \end{tabular}
    \caption{Labeling for Theorem~\ref{2n-5} when $n\equiv 0,2\pmod{3}$ and $n\equiv 1,2,$ or $4\pmod{5}$ }
    \label{tab:2n-5}
\vspace{.2in}
    \centering
    \begin{tabular}{|c|c|c|c|c|c|c|c|}
    \hline
         1& 2 & 3 & $\cdots$ & $n-3$ & $2n-5$ & $2n-1$ & $2n$  \\
    \hline
         $2n-4$ & $2n-7$ & $2n-8$ & $\cdots$ & $n-2$ & $2n-3$ & $2n-2$ & $2n+1$ \\
    \hline
    \end{tabular}
    \caption{Labeling for Theorem~\ref{2n-5} when $n\equiv 0\pmod{3}$ and $n\equiv 3\pmod{5}$}
    \label{tab:2n-5b}
\end{table}

\begin{table}[htb]
    \centering
    \begin{tabular}{|c|c|c|c|c|c|c|c|}
    \hline
         1& 2 & 3 & $\cdots$ & $n-3$ & $2n-5$ & $2n-1$ & $2n$  \\
    \hline
         $2n-6$ & $2n-7$ & $2n-8$ & $\cdots$ & $n-2$ & $2n+1$ & $2n-2$ & $2n-3$ \\
    \hline
    \end{tabular}
    \caption{Labeling for Theorem~\ref{2n-5} when $n\equiv 2\pmod{3}$ and $n\equiv 3\pmod{5}$}
    \label{tab:2n-5c}
\vspace{.2in}
    \centering
    \begin{tabular}{|c|c|c|c|c|c|c|c|c|c|}
    \hline
         1& 2 & 3 & $\cdots$ & $n-5$ & $n-4$ & $n-3$ & $n-2$ & $n+2$ & $n+3$ \\
    \hline
         $n+7$ & $n+8$ & $n+9$ & $\cdots$ & $2n+1$ & $n+1$ & $n$ & $n-1$ & $n+4$ & $n+6$  \\
    \hline
    \end{tabular}
    \caption{Labeling for Theorem~\ref{n+6} when $n\equiv 1$ or $2\pmod{3}$ and $n\equiv 0,2,$ or $3\pmod{5}$}
    \label{tab:n+6}
\vspace{.2in}
    \centering
    \begin{tabular}{|c|c|c|c|c|c|c|c|c|c|}
    \hline
         1& 2 & 3 & $\cdots$ & $n-5$ & $n-4$ & $n-3$ & $n-2$ & $n+2$ & $n+5$ \\
    \hline
         $n+7$ & $n+8$ & $n+9$ & $\cdots$ & $2n+1$ & $n+1$ & $n$ & $n+3$ & $n+4$ & $n+6$  \\
    \hline
    \end{tabular}
    \caption{Labeling for Theorem~\ref{n+6} when $n\equiv 2\pmod{3}$ and $n\equiv 1\pmod{5}$}
    \label{tab:n+6b}
\vspace{.2in}
    \centering
    \begin{tabular}{|c|c|c|c|c|c|c|c|c|c|}
    \hline
         1& 2 & 3 & $\cdots$ & $n-5$ & $n-4$ & $n-3$ & $n-2$ & $n+4$ & $n+5$ \\
    \hline
         $n+7$ & $n+8$ & $n+9$ & $\cdots$ & $2n+1$ & $n+1$ & $n$ & $n+3$ & $n+2$ & $n+6$  \\
    \hline
    \end{tabular}
    \caption{Labeling for Theorem~\ref{n+6} when $n\equiv 1\pmod{3}$ and $n\equiv 1\pmod{5}$}
    \label{tab:n+6c}
\end{table}

\end{document}